\documentclass[12pt]{article}

\usepackage{amsfonts, amsmath, amssymb}
\usepackage{amsthm}
\usepackage[english]{babel}
\usepackage{hyperref}
\newtheorem{theorem}{Theorem}[section]

\newtheorem{lemma}[theorem]{Lemma}

\title{Properties of sets of Subspaces with Constant Intersection Dimension}
\date{\today}
\author{
Lisa Hernandez Lucas
\footnote{Address: Department of Mathematics, Vrije Universiteit Brussel, Pleinlaan 2, 1050 Brussels, Belgium \newline
Email address: \href{mailto:lisa.hernandez.lucas@vub.be}{lisa.hernandez.lucas@vub.be} \newline
Website: \url{http://homepages.vub.ac.be/~lihernan/}
}}

\begin{document}

\maketitle

\begin{abstract}
A $(k,k-t)$-SCID (set of Subspaces with Constant Intersection Dimension) is a set of $k$-dimensional vector spaces that have pairwise intersections of dimension $k-t$. Let $\mathcal{C}=\{\pi_1,\ldots,\pi_n\}$ be a $(k,k-t)$-SCID. Define $S:=\langle \pi_1, \ldots, \pi_n \rangle$ and $I:=\langle \pi_i \cap \pi_j \mid 1 \leq i < j \leq n \rangle$. We establish several upper bounds for $\dim S + \dim I$ in different situations. We give a spectrum result for the case $(n-1)(k-t)\leq k$ and for the case $n\leq\frac{q^{t(n-\eta)}-1}{q^t-1}$, giving examples of $(k,k-t)$-SCIDs reaching a large interval of values for $\dim S + \dim I$.
\end{abstract}

\section{Introduction}

% Definitie SCID in vectorruimte. Verwijzen naar bron waar definitie SCID gegeven is voor projectieve ruimtes.
% Uitleggen waarom dit hetzelfde is als equidistant codes. + bronnen random network coding

Let $\mathcal{V}$ be a vector space over a finite field $\mathbb{F}_q$. Let $k$ and $t$ be integers such that $t \leq k$. A set $\mathcal{C}$ of $k$-dimensional subspaces of $\mathcal{V}$ that have pairwise intersections of dimension $k-t$ is called a $(k,k-t)$-\emph{SCID}. The acronym SCID stands for \emph{a set of Subspaces with Constant Intersection Dimension}.\par
SCIDs are introduced in \cite{SCID}, where a similar definition is given in terms of projective spaces instead of vector spaces. For our purposes however, the ambient space will always be a vector space $\mathcal{V}$ over a finite field $\mathbb{F}_q$ of order $q$. The dimension of this vector space $\mathcal{V}$ does not need to be predefined. \par
In the domain of coding theory, SCIDs are better known as \emph{equidistant codes}.
These codes are relevant in a \emph{random network coding} setting, where information is sent through a network with varying topology. This network is depicted as a directed multigraph where the information has to be transmitted from the \emph{sources} to the \emph{sinks} through some intermediary nodes. Within network coding, these intermediary nodes apply coding to the received inputs, instead of simply routing them. It was shown in \cite{NIF} that the maximal information rate of a network with one source can be achieved by applying this technique. When the order of the ground field is large enough, it is sufficient to only apply \emph{linear} network coding, where the nodes transmit linear combinations of the input they receive \cite{LNC}. 
%When the order of the ground field is large enough, one can achieve maximal information rate by applying \emph{linear} network coding, where the nodes transmit linear combinations of the input they receive \cite{LNC}. 
\emph{Random network coding} is when the nodes output \emph{random} linear combinations of the input, instead of using a predefined scheme. The concept and benefits of \emph{random network coding} in a multi-source setting are explored in \cite{random}, and later on approached mathematically through \emph{subspace codes} in \cite{RNC}. \par
%These codes are relevant for \emph{random network coding}, a concept which is introduced in \cite{RNC}. The use of random network coding is recommended for transmitting information through a network with varying topology, because it can speed up the process [bron].
A subspace code is a code that has vector subspaces as codewords. Note that this is different from the classical case, where the codewords are vectors. The distance between two codewords $U$ and $V$ from a subspace code is called the \emph{subspace distance} and is defined as $d (U,V)=\dim U+\dim V - 2\dim (U\cap V)$. When all codewords have the same dimension, the code is called a \emph{constant dimension code}. When in addition all pairwise intersections have the same dimension, then the distance between any two codewords is constant. In this case we say we have an \emph{equidistant code}, see e.g. \cite{equi}. It should now be clear why these codes correspond to SCIDs. \par
 Note that in the definition of a SCID, it is required that all pairwise intersections have the \emph{same dimension}. It is not necessary that they all coincide. If this is the case, then the SCID is called a $(k,k-t)$-\emph{sunflower}. The $(k-t)$-space that all the elements of the sunflower have in common, is often called the \emph{center} of the sunflower. Sunflowers have been investigated before, see for instance \cite{sunflower}. \par
Another special case of SCIDs occurs when $k=t$, i.e. when every two distinct elements intersect trivially. A $(k,0)$-SCID is called a \emph{partial $k$-spread}.  Note that a partial $k$-spread is also a $(k,0)$-sunflower with trivial center. Partial spreads are studied thoroughly within the domain of finite geometry, see for example \cite{ms} and \cite{seg}.\par

From now on,  we will assume that a SCID contains at least two elements. The following lemma follows directly from the definitions:
\begin{lemma}\label{lemma}
A subset of a $(k,k-t)$-SCID (resp. $(k,k-t)$-sunflower), containing at least two elements, is again a $(k,k-t)$-SCID (resp. $(k,k-t)$-sunflower).
\end{lemma}

Let $\mathcal{C}=\{\pi_1,\ldots,\pi_n\}$ be a $(k,k-t)$-SCID with $n$ elements. Define the following two spaces:
\begin{align*}
S&:=\langle \pi_1, \ldots, \pi_n \rangle  ,\\
I&:=\langle \pi_i \cap \pi_j \mid 1 \leq i < j \leq n \rangle  
.\end{align*}

Intuitively, when the space $S$ has large dimension, the elements of $\mathcal{C}$ are further apart, causing the dimension of $I$ to be smaller and vice versa. This raises the question: \emph{is it possible to give an upper bound on $\dim{S}+\dim{I}$?} \par
The article is structured as follows: In Section \ref{upp}, we establish several upper bounds for different situations. In Section \ref{spec}, we give a spectrum result for the case $(n-1)(k-t)\leq k$ and for the case $n\leq\frac{q^{t(n-\eta)}-1}{q^t-1}$, giving examples of $(k,k-t)$-SCIDs reaching a large interval of values for $\dim S + \dim I$.

\section{Upper bounds on $\dim S + \dim I$}\label{upp}
In this section, we justify the intuition from the previous section by giving upper bounds on the sum $\dim S + \dim I$. When an upper bound is established on this sum, it is clear that for a large dimension of $S$, the dimension of $I$ must be small, and vice versa. We give several upper bounds and compare them for different values of $n \geq 2$, $k$ and $t$. At the end of this section, a summary of the best bounds is given in Table \ref{tabel}.\par
Theorem \ref{main} gives a bound that is valid for all values of $n \geq 2$, $k$ and $t$. \par

\begin{theorem}\label{main}
Let $\mathcal{C}=\{\pi_1,\ldots,\pi_n\}$ be a $(k,k-t)$-SCID, $n\geq 2$. Define $S:=\langle \pi_1, \ldots, \pi_n \rangle$ and $I:=\langle \pi_i \cap \pi_j  \mid 1 \leq i < j \leq n \rangle$. Then:
\[
\dim S + \dim I \leq nk
.\]
\end{theorem}
\begin{proof}
The proof is by induction on the number of spaces $n$. For the induction base, assume $n=2$. Then $\mathcal{C}=\{\pi_1,\pi_2\}$. Hence, $S=\langle \pi_1,\pi_2\rangle$ has dimension $k+t$ and $I=\pi_1\cap\pi_2$ has dimension $k-t$. In this case, $\dim S + \dim I=2k$, agreeing with the theorem. \par
Now assume the theorem is true for $n-1$. Define $\mathcal{C'}:=\{\pi_1,\ldots,\pi_{n-1}\}$, then $\mathcal{C'}$ is a $(k,k-t)$-SCID by Lemma \ref{lemma}. Hence, if we define $S':=\langle \pi_1, \ldots, \pi_{n-1} \rangle$ and $I':=\langle \pi_i \cap \pi_j \mid  1 \leq i < j \leq n-1 \rangle$, then $\dim S' + \dim I' \leq (n-1)k$, by the induction hypothesis. \par
Define $A:=\langle\pi_1 \cap \pi_n,\pi_2\cap\pi_n,\ldots,\pi_{n-1}\cap\pi_n\rangle$, so $A$ is the space spanned by all intersections of $\pi_1,\ldots,\pi_{n-1}$ with the space $\pi_n$. Then $k-t \leq \dim A \leq k$, since $\pi_1\cap\pi_n \subseteq A \subseteq \pi_n$. Note that $I = \langle I',A \rangle$, such that:
\[
\dim I = \dim I' + \dim A - \dim (A \cap I')
.\]
Let $0 \leq \delta \leq t$ be such that $\dim A = k-t+\delta$, then:
\begin{align}\label{I}
\dim I \leq \dim I' + k-t+\delta
.\end{align}
On the other hand, from $S=\langle S', \pi_n \rangle$, it follows:
\[
\dim S = \dim S' + \dim \pi_n - \dim (\pi_n \cap S')
.\]
But $A \subseteq \pi_n \cap S'$, hence $\dim (\pi_n \cap S') \geq \dim A = k-t+\delta$. Together with $\dim \pi_n=k$, this results in the following inequality:
\begin{align}\label{S}
\dim S \leq \dim S' + k - (k-t+\delta) = \dim S' + t -\delta
.\end{align}
Combining (\ref{I}) and (\ref{S}) with the induction hypothesis, we find:
\begin{align*}
\dim S + \dim I & \leq \dim S' + t -\delta + \dim I' + k-t+\delta \\
& \leq (n-1)k + k\\
& \leq nk
,\end{align*}
which concludes the proof.
\end{proof}

The natural question that arises now is whether this bound is sharp. In Theorem \ref{construction}, a construction of a SCID reaching this upper bound is given, under the assumption that $(n-1)(k-t)\leq k$. Hence, under this assumption, the bound given in Theorem \ref{main} is sharp. \par

\begin{theorem}\label{construction}
Let $\mathcal{C}=\{\pi_1,\ldots,\pi_n\}$ be a $(k,k-t)$-SCID. Define $S:=\langle \pi_1, \ldots, \pi_n \rangle$ and $I:=\langle \pi_i \cap \pi_j \mid  1 \leq i < j \leq n \rangle$. Let $(n-1)(k-t)\leq k$. Then
\[
\dim S + \dim I = nk
\]
if and only if there exist $(k-t)$-spaces $V_{ij}$, for each $1\leq i<j\leq n$, and $(k-(n-1)(k-t))$-spaces $U_i$, for each $1\leq i \leq n$, such that the following conditions hold:
\begin{enumerate}
\item For each $1\leq i < j \leq n$, $V_{ij}=\pi_i \cap \pi_j$.
\item For each $1 \leq i \leq n$, $\pi_i = \langle U_i, \pi_i \cap \pi_j \mid  i\neq j\rangle$.
\item The dimension of the span of all the spaces above is maximal, i.e.,
\begin{align*}
&\dim \langle U_i, V_{lj} \mid  i \in \{1,\ldots, n\} \text{ and } 1\leq l < j \leq n \rangle \\
&=n(k-(n-1)(k-t))+\frac{n(n-1)}{2}(k-t).
\end{align*} 
\end{enumerate} 
\end{theorem}
\begin{proof}
First note that to find the dimension of the space in the third condition, we can just sum up the dimensions of the spaces $U_i$ and $V_{lj}$. This gives us the formula in the third condition:
\begin{align*}
\sum_{i=1}^{n}{\dim U_i} + \sum_{1\leq i < j \leq n}{\dim V_{ij}} = n(k-(n-1)(k-t))+\frac{n(n-1)}{2}(k-t)
.\end{align*}
\par
For the first part of the proof, assume that all the enlisted conditions hold for $\mathcal{C}$. Then we want to prove that $\dim S + \dim I=nk$. The third condition implies that the span of all spaces $V_{ij}$, with $1\leq i < j \leq n$, must be maximal. So to find the dimension of $I$, we just need to sum up the dimensions of the spaces $V_{ij}$:
\begin{align}\label{I2}
\dim I = \sum_{1\leq l < j \leq n}{\dim V_{lj}}=\frac{n(n-1)}{2}(k-t)
.\end{align}
On the other hand, the first two conditions imply that $S=\langle \pi_1, \ldots, \pi_n \rangle=\langle U_i, V_{lj} \mid  i \in \{1,\ldots, n\} \text{ and } 1\leq l < j \leq n \rangle$. The third condition immediately implies:
\begin{align}\label{S2}
\dim S = n(k-(n-1)(k-t))+\frac{n(n-1)}{2}(k-t)
.\end{align}
Combining (\ref{I2}) and (\ref{S2}), we get that $\dim S+\dim I=nk$.
%the following equality:
%\begin{align*}
%\dim S + \dim I &= \frac{n(n-1)}{2}(k-t)+n(k-(n-1)(k-t))+\frac{n(n-1)}{2}(k-t) \\
%&=nk
%.\end{align*}
This completes the first part of the proof. \par
The remainder of the proof is again by induction on the size $n$ of $\mathcal{C}$. For the induction base, assume $n=2$. Then it is not hard to see that the enlisted conditions must hold. \par
Now assume that the theorem is true for any $(k,k-t)$-SCID with size $n-1$ and that we have a $(k,k-t)$-SCID $\mathcal{C}=\{\pi_1,\ldots,\pi_n\}$ of size $n$ such that $\dim S + \dim I=nk$. Now define $\mathcal{C'}:=\{\pi_1,\ldots,\pi_{n-1}\}$, $S':=\langle\pi_1,\ldots,\pi_{n-1}\rangle$ and $I':=\langle\pi_i\cap\pi_j \mid 1\leq i < j \leq n-1\rangle$. Then, by Lemma \ref{lemma}, $\mathcal{C'}$ is a $(k,k-t)$-SCID. \par
Note that we only can have $\dim S + \dim I$ maximal, if equality holds in (\ref{I}) and (\ref{S}) in every induction step of the proof of Theorem \ref{main}. Adding up these two equalities, we find that $\dim S + \dim I = \dim S' + \dim I' + k$, implying that also $\dim S'+\dim I'=(n-1)k$ must be maximal. Applying the induction hypothesis on $\mathcal{C'}$ now gives us the following:
\begin{enumerate}
\item There exist $(k-t)$-spaces $V_{ij}$, for each $1\leq i <j \leq n-1$, such that $V_{ij}=\pi_i\cap\pi_j$.
\item There exist $(k-(n-2)(k-t))$-spaces $U'_i$, for each $1 \leq i \leq n-1$, such that $\pi_i = \langle U'_i, \pi_i \cap \pi_j  \mid i\neq j\rangle$.
\item The span of the spaces above has maximal dimension.
% I.e. 
%\begin{align*}
%&\dim \langle U'_i, V'_{lj} \mid  i \in \{1,\ldots, n-1\} \text{ and } 1\leq l < j \leq n-1 \rangle \\
%&=(n-1)(k-(n-2)(k-t))+\frac{(n-1)(n-2)}{2}(k-t).
%\end{align*} 
\end{enumerate}
Define $A:=\langle\pi_1 \cap \pi_n,\pi_2\cap\pi_n,\ldots,\pi_{n-1}\cap\pi_n\rangle$, and $\delta \geq 0$ such that $\dim A = k-t+\delta$.
As remarked before, equality must hold in (\ref{I}) in the proof of Theorem \ref{main}:
\[
\dim I = \dim I'+k-t+\delta
,\]
hence, $\dim (A \cap I')=0$. Now define $V_{in}:=\pi_i\cap\pi_n$, for all $1\leq i <n$. Then $\dim(A\cap I')=0$ implies that $V_{in}\subseteq U'_i$, so the third property implies that the span $\langle V_{ij} \mid 1\leq i < j \leq n \rangle$ has maximal dimension. Analogously as in the first part of the proof, we now find that $\dim I = \frac{n(n-1)}{2}(k-t)$.\par
Now choose $(k-(n-1)(k-t))$-spaces $U_i \subseteq U'_i$, for $1\leq i < n$, such that $U'_i=\langle U_i, V_{in}\rangle$. Also choose a $(k-(n-1)(k-t))$-space $U_n$ such that $\pi_n=\langle V_{in}, U_n \mid 1\leq i < n \rangle$. For these choices of $V_{ij}$, with $1\leq i<j \leq n$, and $U_i$, with $1\leq i \leq n$, the first two conditions are fulfilled. \par
Note that $S=\langle U_i, V_{lj} \mid  i \in \{1,\ldots, n\} \text{ and } 1\leq l < j \leq n \rangle$, while for the dimension of $S$ we have:
\[
\dim S = nk-\dim I= nk - \frac{n(n-1)}{2}(k-t)
,\]
which is exactly the sum of the dimensions of the spaces $V_{ij}$, with $1\leq i<j \leq n$ and the spaces $U_i$, with $1\leq i \leq n$. This is precisely what the third condition states.
\end{proof}

Note that the condition $(n-1)(k-t)\leq k$ is necessary for the construction in Theorem \ref{construction} to work. Moreover, it follows from the proof that if this condition doesn't hold, then there exists no $(k,k-t)$-SCID with $\dim S + \dim I = nk$. This means that the bound given in Theorem \ref{main} is sharp if and only if the condition $(n-1)(k-t)\leq k$ holds.

The objective now is to gain more insight in what happens if the condition $(n-1)(k-t) \leq k$ doesn't hold. For this purpose, it is useful to consider the case where $n=3$. \par
In the case $2(k-t)>k$, the bound from Theorem \ref{main}, although valid, cannot be sharp. Lemma \ref{three} provides an upper bound that is also valid for all values of $k$ and $t$, but which is in fact an improvement in the case $2(k-t)>k$. \par

\begin{lemma}\label{three}
Let $\mathcal{C}=\{\pi_1,\pi_2,\pi_3\}$ be a $(k,k-t)$-SCID. Define $S:=\langle \pi_1,\pi_2,\pi_3 \rangle$ and $I:=\langle\pi_1\cap\pi_2,\pi_1\cap\pi_3,\pi_2\cap\pi_3\rangle$. Then $\dim S + \dim I \leq 2(k+t)$.
\end{lemma}
\begin{proof}
Define $\epsilon \geq 0$, such that $\dim (\pi_1\cap\pi_2\cap\pi_3)=k-t-\epsilon$. \\
Then:
\begin{align*}
\dim I &= k-t+2\epsilon \\
\dim S & \leq k+t+t-\epsilon,
\end{align*}
such that $\dim S + \dim I \leq k+t+t-\epsilon +k-t+2\epsilon = 2k+t+\epsilon$. Note however that $\epsilon \leq k- (k-t) = t$, such that $\dim S + \dim I \leq 2k+2t$.
\end{proof}

Note that the equality $\dim S+\dim I=2(k+t)$ only occurs when $\epsilon=t$. In that case, we have $\pi_1=\langle \pi_1\cap \pi_2,\pi_1\cap\pi_3\rangle$, $\pi_2=\langle \pi_1\cap \pi_2,\pi_2\cap\pi_3\rangle$ and $\pi_3=\langle \pi_2\cap \pi_3,\pi_1\cap\pi_3\rangle$. \par
Inspired by this lemma, we prove a new bound on $\dim S + \dim I$ that is valid for all values of $n \geq 3$, $k$ and $t$. \par

\begin{theorem}
Let $\mathcal{C}=\{\pi_1,\ldots,\pi_n\}$ be a $(k,k-t)$-SCID, with size $n\geq 3$. Define $S:=\langle\pi_1,\ldots,\pi_n\rangle$ and $I:=\langle\pi_i\cap\pi_j \mid 1\leq i<j\leq n\rangle$. Then:
\[
\dim S + \dim I \leq (n-1)k +2t
.\]
\end{theorem}
\begin{proof}
Define $S':=\langle\pi_1,\ldots,\pi_{n-1}\rangle$ and $I':=\langle\pi_i\cap\pi_j \mid 1\leq i<j\leq n-1\rangle$. Then, by Theorem \ref{main}, we know that:
\begin{align}\label{ind}
\dim S' +\dim I' \leq (n-1)k
.\end{align}
Now define $A:=\langle\pi_1\cap\pi_n,\ldots,\pi_{n-1}\cap\pi_n\rangle$ and let $\delta\geq 0$ be such that $\dim A = k-t+\delta$. \par
Note that $I=\langle I',A\rangle$, such that:
\begin{align}\label{Idim}
 \dim I = \dim I' + \dim A - \dim (A \cap I'). 
\end{align} Define $B:=\langle\pi_1\cap\pi_2,\pi_1\cap\pi_3,\ldots,\pi_1\cap\pi_{n-1}\rangle$, the space spanned by all the intersections with $\pi_1$, except for $\pi_1\cap\pi_n$. Then $\pi_1\cap\pi_2\subseteq B$, such that $\dim B \geq k-t$. Moreover we have $\langle \pi_1\cap\pi_n,B\rangle \subseteq \pi_1$, such that:
\begin{center}
	$\dim \pi_1 \geq \dim(\pi_1\cap\pi_n) + \dim B - \dim (\pi_1\cap\pi_n\cap B)$ \\
	$\Downarrow$ \\
	$k \geq k-t + k-t - \dim (\pi_1\cap\pi_n\cap B)$ \\
	$\Downarrow$ \\
	$\dim (\pi_1\cap\pi_n\cap B) \geq k-2t.$
\end{center}
Since $\pi_1\cap\pi_n\cap B\subseteq A\cap I'$, it follows that $\dim( A\cap I' )\geq k-2t$. Combining this with (\ref{Idim}) we get:
\begin{align}\label{I3}
\dim I \leq \dim I' + k-t+\delta-(k-2t)=\dim I'+t+\delta.
\end{align}
On the other hand, $S=\langle S',\pi_n\rangle$ and $A\subseteq S'\cap \pi_n$, such that:
\begin{align}\label{S3}
\dim S \leq \dim S' + \dim \pi_n -\dim A \leq \dim S' + t - \delta.
\end{align}
Combining (\ref{I3}) and (\ref{S3}) with (\ref{ind}), we find:
\begin{align*}
\dim S + \dim I &\leq \dim S' + t -\delta + \dim I' + t + \delta \\
& \leq (n-1)k + 2t.
\end{align*}
\end{proof}

Comparing this new bound to the bound given in Theorem \ref{main}, we can now distinguish three cases:
\begin{itemize}
	\item[$k<2t$] In this case, $nk < (n-1)k+2t$, so the bound from Theorem \ref{main} is the best bound we have. By Theorem \ref{construction}, this bound is sharp if and only if the inequality $(n-1)(k-t)\leq k$ holds.
	\item[$k=2t$] Now we have $nk=(n-1)k+2t$, such that the new bound is the same as the bound given in Theorem \ref{main}. Note that in this case $(n-1)(k-t)\leq k \Leftrightarrow n \leq 3$, such that the bound is only sharp for $n \leq 3$. For $n>3$, we will show in Theorem \ref{ugly} that $\dim S + \dim I \leq nk-(n-3)$. We don't know whether this bound is sharp.
	\item[$k>2t$] Then $(n-1)k+2t < nk$, such that the new bound is an improvement compared to Theorem \ref{main}. In Theorem \ref{ugly}, we will show that $\dim S + \dim I \leq 2k + 2(n-2)t-(n-3)$. Note that, given $k>2t$, we have for $n\geq 3$:
	\begin{align*}
	2k + 2(n-2)t-(n-3) &= 2k + (n-3)2t + 2t - (n-3)\\
	& < 2k + (n-3)k + 2t \\
	& = (n-1)k+2t,
	\end{align*}
	such that Theorem \ref{ugly} indeed gives us a better bound for $n\geq 3$.
\end{itemize}

\begin{theorem}\label{ugly}
For $n\geq 3$, let $\mathcal{C}=\{\pi_1,\ldots,\pi_n\}$ be a $(k,k-t)$-SCID, with $k\geq 2t$. Define $S:=\langle\pi_1,\ldots,\pi_n\rangle$ and $I:=\langle\pi_i\cap\pi_j \mid 1\leq i<j\leq n\rangle$. Then:
\[
\dim S + \dim I \leq 2k+2(n-2)t-(n-3)
.\]
\end{theorem}
\begin{proof}
The proof is by induction on $n$. \par
For the induction base, consider $n=3$. If $k>2t$, it follows from Lemma \ref{three} that $\dim I + \dim S \leq 2k+2t$, agreeing with the theorem. If $k=2t$, then we have by Theorem \ref{main} that $\dim S + \dim I \leq 3k=2k+2t$. \par
Now assume that the theorem is true for $(k,k-t)$-SCIDs with $n-1$ elements. Then it is in particular true for $\mathcal{C'}:=\{\pi_1,\ldots,\pi_{n-1}\}$. Define $S':=\langle\pi_1,\ldots,\pi_{n-1}\rangle$ and $I':=\langle\pi_i\cap\pi_j \mid 1\leq i < j\leq n-1\rangle$. Then the induction hypothesis implies:
\begin{align}\label{ind2}
\dim S' + \dim I' \leq 2k+2(n-3)t-(n-4)
.\end{align}
Define $A:=\langle\pi_1\cap\pi_n,\ldots,\pi_{n-1}\cap\pi_n\rangle$ to be the space spanned by the intersections with the space $\pi_n$. Then $\dim A = k-t+\delta$, for a certain value $\delta \geq 0$. Note that $S=\langle S', \pi_n\rangle$ and that $A\subseteq \pi_n\cap S'$, thus we have:
\begin{align}\label{S5}
\dim S \leq \dim S' + \dim \pi_n - \dim A = \dim S' +t -\delta.
\end{align} 
We can now repeat the same argument as in the previous proof, to find that $\dim (A \cap I') \geq k-2t$. We distinguish between two cases:
\begin{itemize}
	\item Case 1: $\dim (A \cap I') = k-2t$. \par
	For any $i, j$, with $1\leq i < j \leq n-1$, we have $\langle\pi_i\cap\pi_n,\pi_j\cap\pi_n\rangle\subseteq\pi_n$, implying:
	\begin{align*}
	\dim \pi_n &\geq \dim (\pi_i\cap\pi_n) + \dim (\pi_j\cap\pi_n)-\dim (\pi_i\cap\pi_j\cap\pi_n) \\
	\Rightarrow k &\geq k-t+k-t-\dim (\pi_i\cap\pi_j\cap\pi_n).
	\end{align*}
	We find $\dim (\pi_i\cap\pi_j\cap\pi_n)\geq k-2t$. Since $\dim(A\cap I')=k-2t$ and $\pi_i\cap\pi_j\cap\pi_n\subseteq A \cap I'$, we have that $\pi_i\cap\pi_j\cap\pi_n=A\cap I'$. Since this argument is independent from the choices of $i$ and $j$, it follows that all intersections $\pi_i\cap\pi_j\cap\pi_n$ must coincide, for $1\leq i < j \leq n-1$. Hence, the intersections $\pi_i\cap\pi_n$ form a $(k-t,k-2t)$-sunflower inside $\pi_n$.\par
	Now let $X$ be a $t$-dimensional space skew to $\pi_n$, such that $\pi_1\cap\pi_2=\langle\pi_1\cap\pi_2\cap\pi_n,X\rangle$. Then, $\pi_1=\langle\pi_1\cap\pi_n,X\rangle\subseteq\langle\pi_n,X\rangle$ and $\pi_n\subseteq\langle\pi_n,X\rangle$. For any $i$, with $1<i<n$, we have $\pi_i\cap\pi_1\subseteq\pi_i$, $\pi_i\cap\pi_n\subseteq\pi_i$ and
	\begin{align*}
	\dim\langle\pi_i\cap\pi_1,\pi_i\cap\pi_n\rangle &= \dim(\pi_i\cap\pi_1)+\dim(\pi_i\cap\pi_n)-\dim(\pi_1\cap\pi_i\cap\pi_n) \\
	&=k-t+k-t-(k-2t) \\
	&=k.
	\end{align*}
	This implies that $\pi_i=\langle\pi_i\cap\pi_1,\pi_i\cap\pi_n\rangle\subseteq\langle\pi_1,\pi_n\rangle\subseteq\langle\pi_n,X\rangle$. Hence, $S \subseteq \langle\pi_n,X\rangle$ and $I \subseteq \langle\pi_n,X\rangle$. \par
	We have that $\dim\langle\pi_n,X\rangle=k+t$, which implies that $\dim I \leq k+t$ and $\dim S \leq k+t$, such that $\dim S + \dim I \leq 2(k+t)$. This bound is lower than the one stated in the theorem.
	\item Case 2: $\dim (A \cap I') > k-2t$. \par
	We have that $I=\langle I',A\rangle$, such that
\begin{align}\label{I5}
\dim I &= \dim I'+ \dim A - \dim (A \cap I') \nonumber \\
& < \dim I' + k-t+\delta-(k-2t) \nonumber \\
&<\dim I'+t+\delta.
\end{align}
Combining (\ref{S5}) and (\ref{I5}) with (\ref{ind2}) we get:
\begin{align*}
\dim S + \dim I &\leq \dim S' + t - \delta + \dim I'+t+\delta -1\\
& \leq 2k+2(n-3)t-(n-4) + 2t-1 \\
& \leq 2k+2(n-2)t-(n-3).
\end{align*}
\end{itemize}
\end{proof}

We conclude this section with a summary of the bounds for different values of $n$, $k$ and $t$, given by Table \ref{tabel}. \par
\begin{table}
\centering
\begin{tabular}{|c|c|c|}
			\hline
			Condition&Upper bound for $\dim S + \dim I$&Sharpness?\\
			\hline
			\hline
			$(k-t)(n-1) \leq k$ & $nk$ & yes \\
			\hline
			$k\geq 2t$, $n\geq 3$ & $2k+2(n-2)t-(n-3)$ & unknown \\
			\hline
			\begin{tabular}{@{}c@{}}$k<2t$ \\ $(k-t)(n-1) > k$\end{tabular} & $nk$ & no \\
			\hline
\end{tabular}
\caption{Summary of the best bounds found for $\dim S +\dim I$, for different values of $n$, $k$ and $t$.}
\label{tabel}
\end{table}

%\textbf{Summary}
%\begin{enumerate}
%\item If $(k-t)(n-1) \leq k$, then we have $\dim S + \dim I \leq nk$ and this bound is sharp.
%\item If $k>2t$, then we have $\dim S + \dim I \leq 2k+2(n-2)t-(n-3)$. We don't know whether this bound is sharp.
%\item If $k<2t$ and $(k-t)(n-1) > k$, then the best upper bound we found is $\dim S + \dim I <nk$. Note that this bound cannot be sharp.
%\end{enumerate}

\section{Spectrum results}\label{spec}
In this section we construct examples of SCIDs for several values of $\dim S + \dim I$. In the second subsection, we'll assume that the condition $(n-1)(k-t)\leq k$ holds, and adapt the construction from Theorem \ref{construction} to construct new SCIDs with $\dim S + \dim I = nk-\epsilon$. Then later in the third subsection, we will drop the assumption $(n-1)(k-t)\leq k$ and use field reduction to construct sunflowers, which are particular examples of SCIDs, for even smaller values of $\dim S + \dim I$. The concept of field reduction is explained in the first subsection. We finish this section with a summary of the spectrum results we obtained.\par
\subsection{Field reduction}
Field reduction is a powerful tool in finite geometry. The method is described for projective spaces in \cite{mbs} and for projective spaces and polar spaces in \cite{fieldred}. For our purposes however, we will consider field reduction in vector spaces. \par
The idea relies on the fact that $\mathbb{F}_{q^t}$ is a $t$-dimensional vector space over $\mathbb{F}_q$. Hence, all 1-dimensional subspaces of a vector space $V(n,q^t)$ correspond to $t$-dimensional subspaces of the vector space $V(nt,q)$. Moreover, the set of all 1-dimensional spaces of $V(n,q^t)$ corresponds to a $t$-spread in $V(nt,q)$. This spread is often called a \emph{Desarguesian spread}. \par
Note that a $d$-dimensional subspace in $V(n,q^t)$ corresponds to a $dt$-dimensional subspace in $V(nt,q)$. Hence, we have the following lemma:

\begin{lemma}\label{fred}
Using field reduction, a set of 1-dimensional subspaces in $V(n,q^t)$ spanning a $d$-dimensional space corresponds to a partial $t$-spread in $V(nt,q)$ spanning a $dt$-dimensional space.
\end{lemma}

\subsection{Spectrum result on SCIDs}
In the proof of Theorem \ref{construction}, we constructed a SCID such that both $\dim I$ and $\dim S$ were maximal. This way the sum $\dim S + \dim I$ was maximal as well. In this section we want to construct SCIDs with smaller values for this sum. In order to do this, we adapt the construction from Theorem \ref{construction} in such a way that $\dim I$ decreases, while $\dim S$ stays as large as possible.  \par
The idea behind the proof of Theorem \ref{spec1} is that instead of having all pairwise intersections span a space of maximal dimension, we now demand that there is some overlap between three spaces of the SCID. This causes $\dim I$ to decrease. On the other hand, we want $\dim S$ to be as large as possible under this requirement. So apart from the overlap, we need all spaces of the SCID to span maximal dimension. \par
\begin{theorem}\label{spec1}
If $(n-1)(k-t)\leq k$ and $0\leq \epsilon \leq k-t$, then there exists a $(k,k-t)$-SCID $\{\pi_1,\ldots,\pi_n\}$, such that 
\[
\dim S + \dim I =nk - \epsilon,
\] with $S:=\langle\pi_1,\ldots,\pi_n\rangle$ and $I:=\langle\pi_i\cap\pi_j \mid i\neq j\rangle$.
\end{theorem}
\begin{proof}
If $\epsilon=0$, then a construction is given by Theorem \ref{construction}. \par
For $\epsilon>0$, consider a $(k,k-t)$-SCID $\mathcal{C}=\{\pi_1,\ldots,\pi_n\}$ constructed as in the proof of Theorem \ref{construction}. We will slightly adapt this SCID in order to construct a new SCID $\mathcal{C'}=\{\pi'_1,\ldots,\pi'_n\}$. In this new SCID, there will be an $\epsilon$-dimensional overlap between the three spaces $\pi'_1$, $\pi'_2$ and $\pi'_n$. \par
Consider an $\epsilon$-dimensional subspace $E \subseteq \pi_1\cap\pi_n$ and two $(k-t-\epsilon)$-dimensional subspaces $D_1 \subseteq \pi_1\cap\pi_2$ and $D_2 \subseteq \pi_2\cap\pi_n$. From the conditions in Theorem \ref{construction}, it follows that $\langle E,D_1,D_2\rangle$ has maximal dimension. We'll choose our spaces from $\mathcal{C'}$ such that the space $E$ is the overlap between $\pi'_1$, $\pi'_2$ and $\pi'_n$, and such that $\pi'_1\cap\pi'_2=\langle E, D_1\rangle$ and $\pi'_2\cap\pi'_n=\langle E, D_2\rangle$. \par
By shifting the intersections like this, we loose $\epsilon$ dimensions from each of the three spaces $\pi_1$, $\pi_2$ and $\pi_n$. To compensate for this loss, choose $\epsilon$-dimensional spaces $P_1$, $P_2$ and $P_n$, such that these spaces span maximal dimension together with the elements of $\mathcal{C}$, i.e., $\dim \langle P_1,P_2,P_n, \pi_i \mid 1 \leq i \leq n \rangle$ is maximal. \par
Remember that by Theorem \ref{construction}, there exist $(k-(n-1)(k-t))$-dimensional spaces $U_1,\ldots,U_n$, such that for each $1\leq i\leq n$, the space $U_i$ is skew to $\langle\pi_1,\ldots,\pi_{i-1},\pi_{i+1},\ldots,\pi_n\rangle$ and such that $\pi_i=\langle U_i, \pi_i\cap\pi_j \mid i\neq j \rangle $. Now we have all components to define the spaces of $\mathcal{C}'$:
\begin{align*}
\pi'_1 &:= \langle D_1,P_1,U_1,\pi_1\cap\pi_j \mid 3\leq j \leq n\rangle \\
\pi'_2 &:= \langle E,D_1,D_2,P_2,U_2,\pi_2\cap\pi_j \mid 3\leq j < n\rangle \\
\pi'_j&:=\pi_j \text{, for } 3\leq j < n \\
\pi'_n &:= \langle D_2,P_n,U_n,\pi_1\cap\pi_n,\pi_j\cap\pi_n \mid 3\leq j < n\rangle.
\end{align*}
Now we first want to show that this defines in fact a $(k,k-t)$-SCID. For the dimensions of the spaces $\pi'_1,\ldots,\pi'_n$, note that we can just sum up the dimensions of the spaces in their definitions above. This follows directly from the third condition in Theorem \ref{construction}. We now find: 
\begin{align*}
\dim \pi'_1 &= (k-t-\epsilon) + \epsilon + (k-(n-1)(k-t)) + (n-2)(k-t)=k \\
\dim \pi'_2 &= \epsilon + 2(k-t-\epsilon) + \epsilon + (k-(n-1)(k-t)) + (n-3)(k-t) = k \\
\dim \pi'_j&=\dim \pi_j = k \text{, for } 3\leq j < n \\
\dim \pi'_n&=(k-t-\epsilon) + \epsilon + (k-(n-1)(k-t)) + (k-t)+(n-3)(k-t)=k.
\end{align*}
Moreover, for the intersections we have that $\pi'_i\cap\pi'_j=\pi_i\cap\pi_j$, for $1\leq i<j\leq n$, except for $\pi'_1\cap\pi'_2$ and $\pi'_2\cap\pi'_n$. For these intersections we have that $\pi'_1\cap\pi'_2=\langle D_1,E\rangle$ and $\pi'_2\cap\pi'_n=\langle D_2,E\rangle$. Note that $E\subseteq \pi_1\cap \pi_n$. Hence all pairwise intersections of the spaces $\pi'_1,\ldots,\pi'_n$ have dimension $k-t$. This implies that $\mathcal{C'}=\{\pi'_1,\ldots,\pi'_n\}$ is indeed a $(k,k-t)$-SCID. \par
Now define $S:=\langle\pi'_1,\ldots,\pi'_n\rangle$ and $I:=\langle\pi'_i\cap\pi'_j \mid 1\leq i<j\leq n\rangle$. Note that, since $E\subseteq\pi'_1\cap\pi'_n$:
\[
I=\langle D_1,D_2,\pi'_i\cap\pi'_j\rangle
,\]
where $(i,j)\in\{1,\ldots,n\}^2\setminus \{(1,2),(2,n)\}$. By the definitions above and the third condition of Theorem \ref{construction}, we can sum up the dimensions of these spaces to find the dimension of $I$: 
\begin{align}\label{I6}
\dim I &= (k-t-\epsilon) + (k-t-\epsilon) + \left(\frac{n(n-1)}{2}-2\right)(k-t)  \nonumber \\
& = \frac{n(n-1)}{2}(k-t)-2\epsilon.
\end{align}
On the other hand, note that:
\[
S  = \langle D_1, D_2, P_1, P_2, P_n, U_l, \pi'_i\cap\pi'_j\rangle,
\]
where $1\leq l \leq n$ and $(i,j)\in\{1,\ldots,n\}^2\setminus \{(1,2),(2,n)\}$. Note again that $E\subseteq \pi'_1\cap\pi'_n$. Again by the third condition of Theorem \ref{construction} 	and by the way we defined these spaces, we can find the dimension of $S$ by summing the dimensions:
\begin{align}\label{S6}
\dim S &= 2(k-t-\epsilon) + 3\epsilon + n(k-(n-1)(k-t)) + \left(\frac{n(n-1)}{2}-2\right)(k-t) \nonumber\\
&=\frac{n(n-1)}{2}(k-t)+\epsilon+ n(k-(n-1)(k-t))
.\end{align}
Combining (\ref{I6}) and (\ref{S6}), we get:
\[\dim S + \dim I =nk-\epsilon.\]
%\begin{align*}
%\dim S + \dim I =& \frac{n(n-1)}{2}(k-t)+\epsilon+ n(k-(n-1)(k-t)) \\ &+ \frac{n(n-1)}{2}(k-t)-2\epsilon \\
%=& nk-\epsilon.
%\end{align*}
We can conclude that $\mathcal{C'}$ is a $(k,k-t)$-SCID fulfilling the condition of the theorem.
\end{proof}

The idea of this proof can be generalized to construct SCIDs with lower values of the sum $\dim S + \dim I$. In the first part of the proof of Theorem \ref{spec2}, we let intersections \emph{coincide} instead of just having some overlap. This again causes $\dim I$ to decrease. Meanwhile we keep $\dim S$ as large as possible by choosing spaces that span maximal dimension. \par
In the second part of the proof, we adapt the construction from the first part by using a similar technique to Theorem \ref{spec1}. \par
\begin{theorem}\label{spec2}
If $(n-1)(k-t)\leq k$, $0\leq \epsilon \leq k-t$ and $2\leq\eta \leq n-1$, then there exists a $(k,k-t)$-SCID $\{\pi_1,\ldots,\pi_n\}$ such that:
\[
\dim S + \dim I =nk -(\eta-2)(k-t)- \epsilon
,\]
with $S:=\langle\pi_1,\ldots,\pi_n\rangle$ and $I:=\langle\pi_i\cap\pi_j \mid i\neq j\rangle$.
\end{theorem}
\begin{proof}
For $\eta=2$, this is exactly Theorem \ref{spec1}. So assume $\eta>2$. We distinguish between two cases, based on the value of $\epsilon$. \par
\textbf{Case 1: $\epsilon=0$} \newline
Let $\mathcal{C}$ be a $(k,k-t)$-SCID as constructed in Theorem \ref{construction}. Just like in the previous proof, we will make adaptations to $\mathcal{C}$ in order to obtain a new $(k,k-t)$-SCID $\mathcal{C}'=\{\pi'_1,\ldots,\pi'_n\}$ fulfilling the conditions of the theorem. For this new SCID, there will be a $(k-t)$-dimensional subspace that all spaces $\pi'_1,\ldots,\pi'_{\eta-1}$ and $\pi'_n$ have in common.\par
Now define $D:=\pi_1\cap\pi_n$, this will be the common subspace. But if we shift the elements of the SCID in such a way that all spaces $\pi'_1,\ldots,\pi'_{\eta-1}$ and $\pi'_n$ have $D$ in common, then we loose $(k-t)(\eta-2)$ dimensions in each of these spaces. To compensate this, choose $(k-t)(\eta-2)$-dimensional spaces $P_1, \ldots, P_{\eta-1}$ and $P_n$, such that these spaces span maximal dimension with the elements of $\mathcal{C}$, i.e., $\dim\langle P_1,\ldots,P_{\eta-1},P_n,\pi_i\mid 1\leq i\leq n \rangle$ is maximal. \par
Note that, by Theorem \ref{construction}, there exist $(k-(n-1)(k-t))$-dimensional spaces $U_1,\ldots,U_n$, such that for each $1\leq i\leq n$, the space $U_i$ is skew to $\langle\pi_1,\ldots,\pi_{i-1},\pi_{i+1},\ldots,\pi_n\rangle$ and such that $\pi_i=\langle U_i, \pi_i\cap\pi_j \mid i\neq j \rangle $. Now we have all components to define the spaces of $\mathcal{C'}$: \par
\begin{align*}
\pi'_i&=\langle D,U_i,P_i,\pi_i\cap\pi_j\mid \eta\leq j < n\rangle \text{, for } 1\leq i < \eta \\
\pi'_j&=\pi_j \text{, for } \eta \leq j < n \\
\pi'_n&=\langle D,U_n,P_n,\pi_n\cap\pi_j\mid \eta\leq j < n\rangle
\end{align*}
We now want to show that this indeed defines a $(k,k-t)$-SCID. To find the dimensions of the spaces $\pi'_1,\ldots,\pi'_n$, note that we can simply sum up the dimensions of the spaces in their definitions above. This follows from the third condition in Theorem \ref{construction}. We now find:
\begin{align*}
\dim\pi'_i&=(k-t)+(k-(n-1)(k-t))+(k-t)(\eta-2)+(n-\eta)(k-t)\\ &=k \text{, for } 1\leq i < \eta \\
\dim\pi'_j&=\dim\pi_j=k\text{, for } \eta\leq j < n \\
\dim\pi'_n&=(k-t)+(k-(n-1)(k-t))+(k-t)(\eta-2)+(n-\eta)(k-t)\\ &=k.
\end{align*}
For the pairwise intersections, we have $\pi'_i\cap\pi'_j=\pi_i\cap\pi_j$, if $1\leq i \leq n$ and $\eta\leq j < n$. All other pairwise intersections are equal to the space $D$. We can conclude that all the pairwise intersections of the spaces $\{\pi'_1,\ldots,\pi'_n\}$ have dimension $k-t$, implying $\mathcal{C'}:=\{\pi'_1,\ldots,\pi'_n\}$ is a $(k,k-t)$-SCID.\par
Now define $S:=\langle\pi'_1,\ldots,\pi'_n\rangle$ and $I:=\langle\pi'_i\cap\pi'_j \mid 1\leq i < j\leq n\rangle$. Note that 
\[
I = \langle D, \pi'_i\cap\pi'_j\mid 1\leq i \leq n \text{, } \eta\leq j < n \text{ and } i\neq j\rangle,
\]
where we can add the dimensions of the intersections $\pi'_i\cap\pi'_j$ and $D$, to find $\dim I$. Note that the number of intersections $\pi'_i\cap\pi'_j$ occurring in this expression is
\[
\eta(n-\eta)+\frac{(n-\eta)(n-\eta-1)}{2}.
\]
From this follows the dimension of $I$:
\begin{align}\label{I7}
\dim I &= (k-t)+\left(\eta(n-\eta)+\frac{(n-\eta)(n-\eta-1)}{2}\right)(k-t),
\end{align}
where the first value $k-t$ comes from $\dim D$.
On the other hand, for $S$ we have:
\begin{align*}
S&=\langle D, U_1,\ldots,U_n,P_1,\ldots,P_{\eta-1},P_n,\pi'_i\cap\pi'_j\mid 1\leq i \leq n \text{, } \eta\leq j < n \text{ and } i\neq j\rangle
.\end{align*}
By construction, we can find the dimension of $S$ by summing up all dimensions of the spaces occurring in the expression above. Hence, we find:
\begin{align}\label{S7}
\dim S =& (k-t)+n(k-(n-1)(k-t))+\eta(k-t)(\eta-2) \nonumber \\
& +\left(\eta(n-\eta)+\frac{(n-\eta)(n-\eta-1)}{2}\right)(k-t) 
.\end{align}
Note that $2\eta(n-\eta)+(n-\eta)(n-\eta-1)=n(n-1)-\eta(\eta-1)$.
%\begin{align*}
%2\eta(n-\eta)&+(n-\eta)(n-\eta-1) \\
%&= 2n\eta-2\eta^2+n(n-1)-n\eta+\eta(\eta+1)-n\eta \\
%&=n(n-1)-\eta(\eta-1).
%\end{align*}
Combining this with (\ref{I7}) and (\ref{S7}), we find:
\[\dim S+\dim I = nk-(\eta-2)(k-t).\]
%\begin{align*}
%\dim S &+ \dim I \\
%=& 2(k-t)+n(k-(n-1)(k-t))+\eta(k-t)(\eta-2)  \\
%& +2\left(\eta(n-\eta)+\frac{(n-\eta)(n-\eta-1)}{2}\right)(k-t)  \\
%=& nk + (k-t)\left[2-n(n-1)+\eta(\eta-2)+n(n-1)-\eta(\eta-1)\right] \\
%=&nk-(\eta-2)(k-t)
%\end{align*}
Hence, $\mathcal{C'}$ is a $(k,k-t)$-SCID fulfilling the condition of the theorem. \par
\textbf{Case 2: $\epsilon>0$} \newline
 Let $\mathcal{C'}=\{\pi'_1,\ldots,\pi'_n\}$ be a $(k,k-t)$-SCID as constructed in the previous case for $\epsilon=0$. Let the spaces $D$, $P_i$ and $U_j$, for $i \in \{1,\ldots,\eta-1,n\}$ and $j \in \{1,\ldots,n\}$, be as defined in the first case. We will again adapt $\mathcal{C'}$ to construct a $(k,k-t)$-SCID $\mathcal{C''}=\{\pi''_1,\ldots,\pi''_n\}$ meeting the desired conditions. For this, we generalize the technique used in the proof of Theorem \ref{spec1}. \par
Let $E \subseteq D=\pi'_1\cap\pi'_n$ be an $\epsilon$-dimensional subspace, this will be the overlap between the spaces $\pi''_1,\ldots,\pi''_\eta$ and $\pi''_n$. Next, consider $(k-t-\epsilon)$-dimensional subspaces $D_1\subseteq \pi'_1\cap\pi'_\eta, \ldots, D_{\eta-1} \subseteq \pi'_{\eta-1}\cap\pi'_\eta$ and $D_n\subseteq\pi'_\eta\cap\pi'_n$. We will choose the elements of $\mathcal{C''}$ in such a way that $\pi''_i\cap\pi''_\eta=\langle D_i,E\rangle$, for $i\in\{1,\ldots,\eta-1,n\}$. \par
Shifting the spaces like this again causes a loss of dimensions. Note that we loose $(\eta-1)\epsilon$ dimensions from $\pi'_\eta$. To compensate, choose an $(\eta-1)\epsilon$-dimensional space $Y$, such that $Y$ spans maximal dimension together with the spaces of $\mathcal{C}$, i.e., $\dim \langle Y, \pi_i \mid 1\leq i \leq n\rangle$ is maximal. \par
From each of the first $\eta-1$ spaces and the $n$th space we loose $\epsilon$ dimensions. To compensate, choose $\epsilon$-dimensional spaces $X_1,\ldots,X_{\eta-1}$ and $X_n$, such that these spaces span maximal dimension together with $Y$ and the spaces of $\mathcal{C'}$. I.e. $\dim \langle X_1,\ldots, X_{\eta-1},X_n,Y,\pi'_i \mid 1\leq i\leq n\rangle$ is maximal. \par
Now we have everything we need to define $\mathcal{C''}=\{\pi''_1,\ldots,\pi''_n\}$:
\begin{align*}
\pi''_i&:=\langle D, D_i,U_i, P_i, X_i, \pi'_i\cap\pi'_j\mid \eta<j<n \rangle\text{, for } 1\leq i < \eta \\
\pi''_\eta&:=\langle E,D_1,\ldots,D_{\eta-1},D_n,U_\eta,Y,\pi'_\eta\cap\pi'_j\mid \eta<j<n\rangle \\
\pi''_j&:=\pi'_j\text{, for } \eta<j<n \\
\pi''_n&:=\langle D,D_n,U_n,P_n,X_n,\pi'_j\cap\pi'_n\mid \eta<j<n\rangle.
\end{align*}
By the way we defined the spaces in the definitions above, we can just sum up the dimensions of the subspaces to find the dimensions of the spaces $\pi''_i$:
\begin{align*}
\dim\pi''_i=& (k-t)+(k-t-\epsilon)+(k-(n-1)(k-t))+(\eta-2)(k-t) \\
&+\epsilon+(n-\eta-1)(k-t)\\
=&k\text{, for } 1\leq i <\eta \\
\dim\pi''_\eta=&\epsilon+\eta(k-t-\epsilon)+(k-(n-1)(k-t))+(\eta-1)\epsilon+(n-\eta-1)(k-t) \\
=&k \\
\dim\pi''_j=&\dim\pi'_j=k\text{, for }\eta<j<n \\
\dim\pi''_n=&(k-t)+(k-t-\epsilon)+(k-(n-1)(k-t))+(\eta-2)(k-t) \\
&+\epsilon+(n-\eta-1)(k-t) \\
=&k
\end{align*}
For the pairwise intersections we have $\pi''_i\cap\pi''_j=\pi'_i\cap\pi'_j$, as long as $i$ and $j$ are different from $\eta$. For $\eta<i<n$, we have $\pi''_\eta\cap\pi''_i=\pi'_\eta\cap\pi'_i$. For $1\leq i <\eta$, $\pi''_i\cap\pi''_\eta=\langle E,D_i \rangle$ and similarly we have $\pi''_\eta\cap\pi''_n=\langle E,D_n\rangle$. We can conclude that all the pairwise intersections of the spaces $\{\pi''_1,\ldots,\pi''_n\}$ have dimension $k-t$, implying $\mathcal{C''}:=\{\pi''_1,\ldots,\pi''_n\}$ is indeed a $(k,k-t)$-SCID.\par
Now define $S:=\langle\pi''_1,\ldots,\pi''_n\rangle$ and $I:=\langle\pi''_i\cap\pi''_j\mid 1\leq i < j\leq n\rangle$. Note that, since $E\subseteq D$: 
\[
I = \langle D,D_1,\ldots,D_{\eta-1},D_n,\pi''_i\cap\pi''_j\rangle
,\]
where $1\leq i \leq n$, $\eta<j<n$ and $i\neq j$. Note that the number of different intersections $\pi''_i\cap\pi''_j$ occurring in the expression above is:
\[
(\eta+1)(n-\eta-1)+\frac{(n-\eta-1)(n-\eta-2)}{2}
.\]
By construction, we can again sum the dimensions of the spaces occurring in the expression above to find the dimension of $I$:
\begin{align}\label{I8}
\dim I =& (k-t)+\eta(k-t-\epsilon) \nonumber \\
&+\left((\eta+1)(n-\eta-1)+\frac{(n-\eta-1)(n-\eta-2)}{2}\right)(k-t).
\end{align}
On the other hand, note for $S$ that:
\begin{align*}
S = \langle D,D_1,\ldots,D_{\eta-1},D_n,X_1,\ldots,X_{\eta-1},X_n,Y,P_1,\ldots,P_{\eta-1},P_n,U_1,\ldots,U_n,\pi''_i\cap\pi''_j\rangle
,\end{align*}
where $1\leq i \leq n$, $\eta<j<n$ and $i\neq j$. Again, the construction allows us to calculate the sum of the dimensions of all the spaces in the expression above to find the dimension of $S$. We now have:
\begin{align}\label{S8}
\dim S =& (k-t) + \eta(k-t-\epsilon)+\eta\epsilon+(\eta-1)\epsilon \nonumber \\
&+\eta(\eta-2)(k-t)+n(k-(n-1)(k-t)) \nonumber \\
&+ \left((\eta+1)(n-\eta-1)+\frac{(n-\eta-1)(n-\eta-2)}{2}\right)(k-t).
\end{align} 
Combining (\ref{I8}) and (\ref{S8}) and noting that
\[2(\eta+1)(n-\eta-1)+(n-\eta-1)(n-\eta-2)=n(n-1)-\eta(\eta+1),\]
%\begin{align*}
%2(\eta+1)(n-\eta-1)&+(n-\eta-1)(n-\eta-2) \\
%&= (n-\eta-1)(n+\eta) \\
%&=n(n-1)-\eta(\eta+1)
%,\end{align*}
we find:
\[\dim S + \dim I = nk-(\eta-2)(k-t)-\epsilon. \]
%\begin{align*}
%\dim S &+ \dim I \\
%=& 2\left[(k-t) + \eta(k-t-\epsilon)\right]+\eta\epsilon+(\eta-1)\epsilon  \\
%&+\eta(\eta-2)(k-t)+n(k-(n-1)(k-t))  \\
%&+ 2\left((\eta+1)(n-\eta-1)+\frac{(n-\eta-1)(n-\eta-2)}{2}\right)(k-t) \\
%=&nk+\epsilon\left[-2\eta+\eta+(\eta-1)\right] \\
%&+ (k-t)\left[2(1+\eta)+\eta(\eta-2)-n(n-1)+n(n-1)-\eta(\eta+1) \right] \\
%=&nk-\epsilon+(k-t)\left[(\eta+1)(2-\eta)-\eta(2-\eta) \right] \\
%=&nk-(\eta-2)(k-t)-\epsilon.
%\end{align*}
This shows that $\mathcal{C''}$ meets the desired conditions, finishing the proof.
\end{proof}

As long as the condition $(n-1)(k-t)\leq k$ holds, we have established examples of $(k,k-t)$-SCIDs with $\dim S + \dim I = N$, for any integer $N \in[nk-(n-2)(k-t),nk]$.\par

\subsection{Spectrum result on sunflowers}

What we actually did in the constructions of the previous section, was making $\dim I$ smaller while keeping $\dim S$ as large as possible. This method eventually gives rise to a maximal $(k,k-t)$-sunflower, for the case $\dim S + \dim I = nk-(n-2)(k-t)=2k+(n-2)t$.\par
Note that for any sunflower we have $\dim I = k-t$, which is the smallest possible dimension for $I$. To construct SCIDs with $\dim S + \dim I \leq 2k+(n-2)t$, it is not possible to further reduce $\dim I$. But since we're dealing with a \emph{maximal} sunflower, we can reduce $\dim S$. In that case we still have a sunflower, which is an example of a SCID.\par
From now on, we drop the condition $(n-1)(k-t)\leq k$. The essence of this section lies in field reduction and the following lemma:

\begin{lemma}\label{quot}
The existence of a $(k,k-t)$-sunflower spanning dimension $d$ is equivalent to the existence of a partial $t$-spread in a $(d-k+t)$-dimensional space, spanning that $(d-k+t)$-dimensional space.
\end{lemma}
\begin{proof}
Let $S$ be the space spanned by the elements of the sunflower and let $C$ be its center. Then $\dim S=d$, $\dim C=k-t$, and all elements of the sunflower have dimension $k$. \par
Now consider the quotient space $S/C$. The elements of the sunflower all contain the center $C$, so in the quotient space they have dimension $k-(k-t)=t$. Since all elements of the sunflower have precisely $C$ as their pairwise intersections, their quotient equivalents must intersect trivially. So they must form a partial $t$-spread in $S/C$. \par
Moreover, since the elements of the sunflower span the space $S$, their quotient equivalents must span $S/C$, which has dimension $d-(k-t)=d-k+t$.
\end{proof}

Remember that we are working in a vector space $\mathcal{V}$ over the field $\mathbb{F}_q$, otherwise we cannot apply field reduction.

\begin{theorem}\label{sun}
If $1\leq\eta\leq n-2$ and $n\leq\frac{q^{t(n-\eta)}-1}{q^t-1}$, then there exists a $(k,k-t)$-SCID $\{\pi_1,\ldots,\pi_n\}$ such that:
\[
\dim S + \dim I =2k +(n-2)t-\eta t
,\]
with $S:=\langle\pi_1,\ldots,\pi_n\rangle$ and $I:=\langle\pi_i\cap\pi_j \mid i\neq j\rangle$.
\end{theorem}
\begin{proof}
We will construct a sunflower $\mathcal{C}$ meeting the conditions. Note that for a sunflower, $\dim I=k-t$. To have $\mathcal{C}$ fulfill the equality in the theorem, we must have that $\dim S=k+(n-1)t-\eta t$. \par
By Lemma \ref{quot}, the existence of such a sunflower is equivalent to the existence of a partial $t$-spread in an $(n-\eta)t$-dimensional vector space $V((n-\eta)t,q)$, spanning that space. We can now use field reduction to guarantee the existence of $\mathcal{C}$. \par
Consider the vector space $V(n-\eta,q^t)$. Choose $n$ lines, such that the last $n-\eta$ lines span the complete space $V(n-\eta,q^t)$. Since $n-\eta<n\leq\frac{q^{t(n-\eta)}-1}{q^t-1}$, where $\frac{q^{t(n-\eta)}-1}{q^t-1}$ is the number of lines in $V(n-\eta,q^t)$, this is always possible. By Lemma \ref{fred}, this set of $n$ lines in $V(n-\eta,q^t)$ corresponds to a partial $t$-spread in $V((n-\eta)t,q)$, spanning the whole space. 
\end{proof}

Note that $\frac{q^{t(n-\eta)}-1}{q^t-1}$ is the cardinality of a $t$-spread in an $(n-\eta)t$-dimensional vector space over $\mathbb{F}_q$. By reversing the arguments used in the previous proof, it is clear that there cannot exist a \emph{$(k,k-t)$-sunflower} with $\dim S + \dim I =2k +(n-2)t-\eta t$ if $n>\frac{q^{t(n-\eta)}-1}{q^t-1}$. However, this does not exclude the existence of an example of a $(k,k-t)$-SCID with these parameters. \par

\begin{theorem}
If $1\leq\eta\leq n-2$, $0\leq\epsilon<t$, and $n\leq\frac{q^{t(n-\eta)}-1}{q^t-1}$, then there exists a $(k,k-t)$-SCID $\{\pi_1,\ldots,\pi_n\}$ such that:
\[
\dim S + \dim I =2k +(n-2)t-\eta t+\epsilon
,\]
with $S:=\langle\pi_1,\ldots,\pi_n\rangle$ and $I:=\langle\pi_i\cap\pi_j \mid i\neq j\rangle$.
\end{theorem}
\begin{proof}
For $\epsilon=0$, this is exactly the previous theorem. \par
For $\epsilon>0$, we will prove the existence of a $(k,k-t)$-sunflower meeting the conditions, in a similar way as in Theorem \ref{sun}. Choose $n$ lines such that the last $n-\eta$ lines span the complete space $V(n-\eta,q^t)$. Similarly to the previous proof, we have by Lemma \ref{fred} that this set of $n$ lines in $V(n-\eta,q^t)$ corresponds to a partial $t$-spread $\{\pi_1,\ldots,\pi_n\}$ in $V((n-\eta)t,q)$, spanning this whole space. Now embed this space $V=V((n-\eta)t,q)$ in a vector space $V'=V((n-\eta)t+\epsilon,q)$. Choose an $\epsilon$-dimensional space $E$ in $V'$, intersecting trivially with $V$. Then $\langle V,E\rangle=V'$. Consider a $(t-\epsilon)$-dimensional subspace $U$ of $\pi_1$. Now replace $\pi_1$ by $\pi'_1=\langle U,E\rangle$. Then $\{\pi'_1,\pi_2,\ldots,\pi_n\}$ is a partial $t$-spread, spanning $V'$. By Lemma \ref{quot}, there exists a sunflower meeting the conditions.
\end{proof}

We have now proved that there exists a $(k,k-t)$-SCID (more precisely, a $(k,k-t)$-sunflower) with $\dim S + \dim I =N$, for any integer $N \in [2k,2k+(n-2)t]$, as long as the condition $n\leq\frac{q^{t(n-\eta)}-1}{q^t-1}$ holds. Note that a $(k,k-t)$-SCID with $\dim S +\dim I < 2k$ cannot exist. 

\subsection{Summary}
There exists a $(k,k-t)$-SCID with $n$ elements and with $\dim S + \dim I = N$,
\begin{itemize}
	\item for any integer $N \in[2k+(n-2)t,nk]$, if $(n-1)(k-t)\leq k$.
	\item for any integer $N \in [2k,2k+(n-2)t]$, if $n\leq\frac{q^{t(n-\eta)}-1}{q^t-1}$.
\end{itemize}

\end{document}